\documentclass[a4paper,12pt]{amsart}
\usepackage{fancyhdr}
\usepackage[english]{babel}
\usepackage[utf8x]{inputenc}
\usepackage[T1]{fontenc}
\usepackage{amsthm}
\usepackage{dsfont}
\usepackage{amssymb}
\usepackage{amsmath}
\theoremstyle{plain}
\usepackage{chngcntr}
\usepackage{relsize}
\usepackage{pdfpages}

\newtheorem{Lemma}{Lemma}
\newtheorem{Theorem}[Lemma]{Theorem}
\newtheorem{Proposition}[Lemma]{Proposition}

\usepackage[a4paper,top=3cm,bottom=2cm,left=3cm,right=3cm,marginparwidth=1.75cm]{geometry}

\usepackage{float}
\usepackage{nccmath}
\usepackage{mathtools}
\usepackage{amsfonts}
\usepackage{amsmath}
\usepackage{graphicx}
\usepackage[colorinlistoftodos]{todonotes}
\usepackage[colorlinks=true, allcolors=purple]{hyperref}
\makeatletter
\newcommand*{\rom}[1]{\expandafter\@slowromancap\romannumeral #1@}
\makeatother

\title{Small solutions of ternary quadratic congruences with averaging over the moduli}



\author{Stephan Baier}
\address{Stephan Baier,
Ramakrishna Mission Vivekananda Educational and Research Institute, Department of Mathematics, G. T. Road, PO Belur Math, Howrah, West Bengal 711202, India}
\email{stephanbaier2017@gmail.com}
\author{Aishik Chattopadhyay}
\address{Aishik Chattopadhyay,
Ramakrishna Mission Vivekananda Educational and Research Institute, Department of Mathematics, G. T. Road, PO Belur Math, Howrah, West Bengal 711202, India}
\email{aishik.ch@gmail.com}
\begin{document}
\maketitle
\begin{abstract} 
In a recent paper, we proved that for any large enough odd modulus $q\in \mathbb{N}$ and fixed $\alpha_2\in \mathbb{N}$ coprime to $q$, the congruence
\[
x_1^2+\alpha_2x_2^2+\alpha_3x_3^2\equiv 0 \bmod{q}
\]
has a solution of $(x_1,x_2,x_3)\in \mathbb{Z}^3$ with $x_3$ coprime to $q$ of heighth $\max\{|x_1|,|x_2|,|x_3|\}\le q^{11/24+\varepsilon}$ for, in a sense, almost all $\alpha_3$, where $\alpha_3$ runs over the reduced residue classes modulo $q$. Here it was of significance that $11/24<1/2$, so we broke a natural barrier. In this paper, we average the moduli $q$ in addition, establishing the existence of a solution of heighth $\le Q^{3/8+\varepsilon}\alpha_2^{\varepsilon}$ for almost all pairs $(q,\alpha_3)$, with $Q$ large enough, $Q<q\le 2Q$, $q$ coprime to $2\alpha_2$ and $\alpha_3$ running over the reduced residue classes modulo $q$.  
\end{abstract}

\tableofcontents

\section{Introduction and main results}
In this paper, we continue our study of small solutions of diagonal ternary quadratic congruences 
\begin{equation} \label{ourcons}
\alpha_1x_1^2+\alpha_2x_2^2+\alpha_3x^3\equiv 0 \bmod{q},
\end{equation}
begun in \cite{BaCh} and \cite{BaCh1}. 

The study of quadratic congruences of the form $Q(x_1,...,x_n)\equiv 0\bmod{q}$, with $Q$ an integral quadratic form, has attracted a lot of attention. In particular, if $q$ is odd and squarefree, Heath-Brown \cite[Theorem 2]{HB} proved that for any integral ternary quadratic form $Q(x_1,x_2,x_3)$ with determinant coprime to $q$, there exists a solution $(x_1,x_2,x_3)\in \mathbb{Z}^3\setminus\{(0,0,0)\}$ of height $\max\{x_1,x_2,x_3\}\ll q^{5/8+\varepsilon}$ to the congruence
\begin{equation} \label{congruence}
Q(x_1,x_2,x_3)\equiv 0\bmod{q}
\end{equation}
(see \cite{snu} for a corresponding result regarding the case of prime power moduli).  It may be expected that the exponent $5/8$ in Heath-Brown's result can be replaced by $1/2$. In general it is not possible to lower this exponent further as the example of the congruence $x_1^2+x_2^2+x_3^2\equiv 0 \bmod{q}$ shows. However, we proved in  \cite{BaCh1} that, in a sense, almost all congruences of the form in \eqref{ourcons} modulo a fixed $q$ have a solution whose height is much smaller than $q^{1/2}$.  Precisely, we established the following result (see \cite[first part of Theorem 1 in a slightly modified form]{BaCh1}).  

\begin{Theorem} \label{mainthmprevious1}
Let $\varepsilon>0$, $q\in \mathbb{N}$ be odd and $\alpha_2\in \mathbb{Z}$ such that {\rm gcd}$(\alpha_2,q)=1$. Then for all 
$$\alpha_3\in \mathcal{A}(q):=\{t\in \mathbb{Z}: 1\le t\le q, \ \mbox{\rm gcd}(t,q)=1\}
$$ 
with at most $o(\varphi(q))$ exceptions, the congruence
\begin{equation} \label{congr}
x_1^2+\alpha_2x_2^2+\alpha_3x_3^2\equiv 0 \bmod{q}
\end{equation}
has a solution $(x_1,x_2,x_3)\in \mathbb{Z}^3$ satisfying {\rm gcd}$(x_3,q)=1$ and $\max\{|x_1|,|x_2|,|x_3|\}\le q^{11/24+\varepsilon}=q^{0.458333\cdots+\varepsilon}$, provided that $q$ is large enough. 
\end{Theorem}
This was a direct consequence of the following asymptotic result on the number of solutions to \eqref{congr} of bounded heighth (see \cite[second part of Theorem 1]{BaCh1}). 

\begin{Theorem} \label{mainthmprevious2} 
Keep the assumptions and notations in Theorem \ref{mainthmprevious1}. Let $N$ be a positive real number satisfying $q^{11/24+\varepsilon}\le N\le q$. Then for all  $\alpha_3\in \mathcal{A}(q)$ with at most $o(\varphi(q))$ exceptions, the asymptotic formula
\begin{equation} \label{asympform}
\sum\limits_{\substack{|x_1|,|x_2|,|x_3|\le N\\ \text{\rm gcd}(x_3,q)=1\\ 
x_1^2+\alpha_2x_2^2+\alpha_3x_3^2\equiv 0 \bmod{q}}} 1=C_q\cdot \frac{(2N)^3}{q}\cdot \left(1+o(1)\right)
\end{equation}
holds, where
\begin{equation} \label{Cqdef}
C_q:=\prod\limits_{p|q}\left(1-\frac{1}{p}\right)\cdot \prod\limits_{p|q}\left(1-\frac{1}{p}\cdot \left(\frac{-\alpha_2}{p}\right)\right).
\end{equation}
Here $\left(\frac{\cdot}{p}\right)$ denotes the Legendre symbol. 
Moreover, under the Lindel\"of hypothesis for Dirichlet $L$-functions, the exponent $11/24$ above can be replaced by  $1/3$. 
\end{Theorem}

We note that the said exponent $1/3$ is the limit of what can be achieved since if $N$ is much smaller than $q^{1/3}$, then the main term on the right-hand side of \eqref{asympform} becomes much smaller than 1. 

The goal of this paper is to make unconditional progress on the exponent $11/24$ by varying the moduli $q$ over an interval of the form $(Q,2Q]$ in addition.  We establish the following existence result on small solutions to the congruence in \eqref{congr} for almost all pairs $(q,\alpha_3)$, with $q$ running over a dyadic interval and $\alpha_3$ running over the reduced residue classes modulo $q$. 

\begin{Theorem} \label{mainthm1}
Let $\varepsilon>0$, $Q\ge 1$ and $\alpha_2\in \mathbb{N}$. Define
$$
\mathcal{B}(Q,\alpha_2):=\{(q,t)\in \mathbb{N}^2: Q<q\le 2Q, \ \mbox{\rm gcd}(2\alpha_2,q)=1,\ 1\le t\le q, \ \mbox{\rm gcd}(q,t)=1\}.
$$ 
Then for all $(q,\alpha_3)\in \mathcal{B}(Q,\alpha_2)$
with at most $O\left(Q^{2-\varepsilon}\right)$ exceptions, the congruence in \eqref{congr}
has a solution $(x_1,x_2,x_3)\in \mathbb{Z}^3$ satisfying {\rm gcd}$(x_3,q)=1$ of heighth $\max\{|x_1|,|x_2|,|x_3|\}\le Q^{3/8}(\alpha_2Q)^{\varepsilon}=Q^{0.375}(\alpha_2Q)^{\varepsilon}$, provided that $Q$ is large enough. 
\end{Theorem}

It is easy to see that  $\sharp \mathcal{B}(Q,\alpha_2)\gg Q^2\cdot \varphi(\alpha_2)/\alpha_2\gg Q^2/\log\log(10\alpha_2)$ and hence, the above Theorem \ref{mainthm1} with a small enough $\varepsilon>0$ gives a non-trivial result if $\alpha_2\le Q^A$, where $A$ is any fixed positive number. 

Theorem \ref{mainthm1} is a direct consequence of Theorem \ref{mainthm2} with $\Delta=N/2$ below, which provides an asymptotic result on a weighthed number of solutions to \eqref{congr} of bounded heighth. 

\begin{Theorem} \label{mainthm2} 
Keep the assumptions and notations in Theorem \ref{mainthm1}. Let $\Delta$ and $N$ be positive real numbers satisfying
\begin{equation} \label{DeltaN}
Q^{1/3}N^{1/9}(\alpha_2 Q)^{4\varepsilon}\le \Delta\le N\le Q^{1/2}.
\end{equation}
Let $\Phi:\mathbb{R}\rightarrow \mathbb{R}$ be a smooth function satisfying
\begin{equation} \label{Phishape}
\Phi(x)=\begin{cases} 1 & \mbox{\rm if } 0\le |x|\le N-\Delta,\\
\mbox{\rm monotonic } & \mbox{\rm if } N-\Delta\le |x|\le N+\Delta,\\
0 & \mbox{\rm if } |x|\ge N+\Delta,
\end{cases}
\end{equation}
and 
\begin{equation} \label{derbounds}
\Phi^{(j)}(x)\ll_j \Delta^{-j} \quad \mbox{\rm for all } j\in \mathbb{N}\  \mbox{\rm and } x\in \mathbb{R}. 
\end{equation}
Then for all  $(q,\alpha_3)\in \mathcal{B}(Q,\alpha_2)$ with at most $O\left(Q^{2-\varepsilon}\right)$ exceptions, the asymptotic formula
\begin{equation} \label{asympform'}
\sum\limits_{\substack{(x_1,x_2,x_3)\in \mathbb{Z}^3\\ |x_1|,|x_2|\le N\\ \text{\rm gcd}(x_3,q)=1\\ 
x_1^2+\alpha_2x_2^2+\alpha_3x_3^2\equiv 0 \bmod{q}}} \Phi(x_3)=C_q\cdot \frac{(2N)^2}{q}\cdot \int\limits_{\mathbb{R}} \Phi(x){\rm d}x\cdot \left(1+O\left(Q^{-\varepsilon}\right)\right)
\end{equation}
holds, where $C_q$ is defined as in \eqref{Cqdef}.
\end{Theorem}

In this context, it is desirable to consider the corresponding sum with sharp cutoff in $x_3$ on the left-hand side of \eqref{asympform} as well. We derive the following asymptotic result, for which we use Theorem \ref{mainthm2} above again. 

\begin{Theorem} \label{mainthm3} 
Keep the assumptions and notations in Theorem \ref{mainthm1}. Let $N$ be a positive real number satisfying $Q^{45/104}(\alpha_2Q)^{3\varepsilon}=Q^{0.43269...}(\alpha_2Q)^{3\varepsilon}\le N\le Q^{1/2}$. 
Then for all  $(q,\alpha_3)\in \mathcal{B}(Q,\alpha_2)$ with at most $O\left(Q^{2-\varepsilon}\right)$ exceptions, the asymptotic formula
\begin{equation} \label{newasym}
\sum\limits_{\substack{|x_1|,|x_2|,|x_3|\le N\\ \text{\rm gcd}(x_3,q)=1\\ 
x_1^2+\alpha_2x_2^2+\alpha_3x_3^2\equiv 0 \bmod{q}}} 1=C_q\cdot \frac{(2N)^3}{q}\cdot \left(1+O\left(Q^{-\varepsilon}\right)\right)
\end{equation}
holds.
\end{Theorem}
$ $\\
{\bf Remark 1:} We get a slightly better exponent of $159/368=0.43206...$ in place of $45/104=0.43269...$ if we restrict the moduli $q$ to cubefree numbers in addition (see Remark 2 in section \ref{final}).\\ 

The novelty in this paper, as compared to \cite{BaCh1}, is to use an eighth moment bound for short character sums to moduli in a dyadic interval.  We will be brief at places where the arguments from \cite{BaCh1} repeat. Throughout the sequel, $\varepsilon$ will be an arbitrarily small positive number which, depending on the context, may change from a line to the next. For $n\in \mathbb{N}$, we denote by $\tau(n)$ the number of divisors of $n$. Recall the well-known bound $\tau(n)\ll_{\varepsilon} n^{\varepsilon}$ for every $\varepsilon>0$. \\ \\
{\bf Acknowledgements.} The authors would like to thank the anonymous referee for valuable comments. They also wish to thank Professor Ram Murty for the suggestion to consider an extra averaging over the moduli. Moreover, they would like to thank the Ramakrishna Mission Vivekananda Educational and Research Institute for an excellent work environment. The research of the second-named author was supported by a CSIR Ph.D fellowship under file number 09/0934(13170)/2022-EMR-I.
\section{Preliminaries}
We will use the following estimate for short character sums due to Burgess.
 
\begin{Proposition}[Burgess] \label{Burgess}
Let $M\in \mathbb{R}$, $N\ge 1$ and $\chi$ be a non-principal Dirichlet character modulo $q$. Then
\begin{equation} \label{claimed}
\sum\limits_{M<n\le M+N} \chi(n) \ll_{\varepsilon,r} N^{1-1/r}q^{(r+1)/(4r^2)+\varepsilon}
\end{equation}
for $r=2,3$, and for any $r\in \mathbb{N}$ if $q$ is cubefree. 
\end{Proposition}

\begin{proof}
This is \cite[Theorem 12.6]{IK} if $\chi$ is primitive. If $\chi$ is non-primitive and induced by a primitive character $\chi_1$ of conductor $q_1>1$ dividing $q$, then 
$$
\sum\limits_{M<n\le M+N} \chi(n) = \sum\limits_{\substack{M<n\le M+N\\ \text{gcd}(n,q)=1}} \chi_1(n)=\sum\limits_{d|q} \mu(d)\chi_1(d) \sum\limits_{M/d<m\le M/d+N/d} \chi_1(m)
$$ 
using M\"obius inversion. Now the claimed bound \eqref{claimed} follows by applying the already established bound to the above character sum with a primitive character $\chi_1(m)$.  
\end{proof}

For an integrable function $f:(0,\infty)\to \mathbb{C}$, its Mellin transform is defined as
$$
\hat{f}(s):= \int_{0}^{\infty} f(x) x^{s-1} {\rm d}x
$$
for $s\in \mathbb{C}$, whenever this integral converges. We will use the Mellin inversion formula below to link smoothed character sums to Dirichlet $L$-functions.  

\begin{Proposition}[Mellin inversion formula]\label{MIF}
Suppose $f:(0,\infty)\to\mathbb{C}$ is a continuous function of bounded variation and $a<b$ are real numbers such that the integral defining its Mellin transform $\hat{f}(s)$ converges absolutely for every $s$ in the strip $\{z\in \mathbb{C}: a< \mbox{Re}(z)< b\}$. 
Then for every $x>0$ and $c\in (a,b)$, the inversion formula
$$
f(x)=\frac{1}{2\pi i}\int\limits_{(c)} \hat{f}(s) x^{-s}{\rm d}s
$$
holds, where the integral is taken along the vertical line $\mbox{Re}(s)=c$.
\end{Proposition}
\begin{proof}
See \cite[Page 90]{IK}. 
\end{proof}

We will also need a mild bound for Dirichlet $L$-functions. The convex bound below will be sufficient for our purposes.

\begin{Proposition}[Convex bound for Dirichlet $L$-functions]\label{CB}
Let $q$ be a positive integer and $\chi$ be a non-principal Dirichlet character modulo $q$. Then if $\sigma\ge 0$ and $t\in \mathbb{R}$, we have
\begin{equation} \label{claimed2}
L(\sigma+it,\chi)
\ll_{\varepsilon}
(q(1+|t|))^{\max\{0, (1-\sigma)/2\}+\varepsilon}.
\end{equation}
\end{Proposition}

\begin{proof} Let us first assume that $\chi$ is a primitive character of conductor $q>1$. Then for $0\le \sigma\le 1$, the claimed bound \eqref{claimed2} follows from \cite[Exercise 3 on page 100]{IK}, applied to Dirichlet $L$-functions. For $\sigma\ge 1$, this bound is standard and can be achieved via partial summation. If $\chi$ is non-primitive and induced by a primitive character $\chi_1$ of conductor $q_1>1$ dividing $q$, then 
$$
|L(s,\chi)|=\bigg|\prod\limits_{\substack{p|q\\ p\nmid q_1}} \left(1-p^{-s}\right)\bigg|\cdot |L(s,\chi_1)|\le \tau(q)|L(s,\chi_1)|\ll_{\varepsilon} q^{\varepsilon}|L(s,\chi_1)|
$$ 
if $\sigma=\mbox{Re}(s)\ge 0$, and hence, the claimed bound \eqref{claimed2} remains valid in this case. 
\end{proof}

A key ingredient in this paper is the following hybrid eighth moment bound for Dirichlet $L$-functions. 

\begin{Proposition}\label{eighth} Let $T,Q\ge 1$. Then
$$
\sum\limits_{q\le Q}\ \sum\limits_{\chi\bmod{q}}\ \int\limits_{-T}^T \left|L\left(\frac{1}{2}+it,\chi\right)\right|^8  {\rm d}t\ll
Q^2T^2\log^{16}(2QT).
$$
\end{Proposition}
\begin{proof} This follows from \cite[Proposition 3.2 with $c=0$]{CLMR}.
\end{proof}

Finally, we will need the following bound for the number of representations by a particular binary quadratic form.  

\begin{Proposition} \label{repr}
Let $\alpha\in \mathbb{N}$. Denote by $r(n,\alpha)$ the number of representations of $n\in \mathbb{N}$ in the form 
$$
n=x^2+\alpha y^2 \quad \mbox{with } x,y\in \mathbb{Z}. 
$$
Then 
$$
r(n,\alpha)\le 6\tau(n).
$$
\end{Proposition}

\begin{proof} Let $K:=\mathbb{Q}(\sqrt{-\alpha})$, $\mathcal{O}_K$ be the ring of algebraic integers in $K$ and $\omega_K$ be the number of units in $\mathcal{O}_K$. Then since $x^2+\alpha y^2=(x+y\sqrt{-\alpha})(x-y\sqrt{-\alpha})$, we have 
\begin{equation} \label{resti}
r(n,\alpha)\le \omega_K R(n,\alpha),
\end{equation}
where $R(n,\alpha)$ is the number of ways to write the ideal $(n)\subseteq \mathcal{O}_K$ in the form
\begin{equation} \label{idealrep}
(n)=\mathfrak{a}\overline{\mathfrak{a}}
\end{equation}
with $\mathfrak{a}$ being an integral ideal and $\overline{\mathfrak{a}}$ its conjugate in $\mathcal{O}_K$. 
Suppose that $n$ has the prime factorization
$$
n=\prod\limits_{p|n} p^{e_p}.
$$ 
Let $\mathcal{R}$, $\mathcal{I}$, $\mathcal{C}$ be the sets of primes $p$ dividing $n$ which are ramified, inert, split completely in $\mathcal{O}_K$, respectively. 
Then if $(n)$ has a representation as in \eqref{idealrep}, the exponents $e_p$ with $p\in \mathcal{I}$ are necessarily even, and $(n)$ has a prime ideal factorization of the form
$$
(n)=\prod\limits_{p\in \mathcal{R}} \mathfrak{p}^{2e_p} \prod\limits_{p\in \mathcal{I}} (p)^{e_p} \prod\limits_{p\in \mathcal{C}} \left(\mathfrak{p}^{e_p}\overline{\mathfrak{p}}^{e_p}\right),
$$
where $(p)=\mathfrak{p}^2$ is the prime ideal factorization of $(p)$ if $p\in \mathcal{R}$, and $(p)=\mathfrak{p}\overline{\mathfrak{p}}$ is the prime ideal factorization of $(p)$ if $p\in \mathcal{C}$. Hence, \eqref{idealrep} implies that $\mathfrak{a}$ is of the form
$$
\mathfrak{a}=\prod\limits_{p\in \mathcal{R}} \mathfrak{p}^{e_p} \prod\limits_{p\in \mathcal{I}} (p)^{e_p/2} \prod\limits_{p\in \mathcal{C}} \left(\mathfrak{p}^{f_p}\overline{\mathfrak{p}}^{e_p-f_p}\right),
$$
and $\overline{\mathfrak{a}}$ is of the form
$$
\overline{\mathfrak{a}}=\prod\limits_{p\in \mathcal{R}} \mathfrak{p}^{e_p} \prod\limits_{p\in \mathcal{I}} (p)^{e_p/2} \prod\limits_{p\in \mathcal{C}} \left(\mathfrak{p}^{e_p-f_p}\overline{\mathfrak{p}}^{f_p}\right),
$$
where $0\le f_p\le e_p$ if $p\in \mathcal{C}$. Thus, the number of choices $R(n,\alpha)$ of $\mathfrak{a}$ equals 
\begin{equation} \label{Reval}
R(n,\alpha)= \prod\limits_{p\in \mathcal{C}} (1+e_p)\le \prod\limits_{p|n} (1+e_p)= \tau(n).
\end{equation}
Now the result follows from \eqref{resti} and \eqref{Reval} since $\omega_K\le 6$. 
\end{proof} 

Tacitly, we will use the well-known bound (see \cite[Theorem 328]{HW})
$$
\frac{q}{\varphi(q)}\ll \log\log(10q) 
$$
at several places in this article. 

\section{Simplification of the problem} \label{ini}
The initial steps in our proof of Theorem \ref{mainthm2} are similar to those in \cite{BaCh1}. Set 
$$
S(q,\alpha_3):=\sum\limits_{\substack{(x_1,x_2,x_3)\in \mathbb{Z}^3\\ |x_1|,|x_2|\le N\\ (x_3,q)=1\\ 
x_1^2+\alpha_2x_2^2+\alpha_3x_3^2\equiv 0 \bmod{q}}} \Phi(x_3)
$$
and 
$$
M(q):=\frac{1}{\varphi(q)} 
\sum_{\substack{(x_1,x_2,x_3)\in \mathbb{Z}^3\\|x_1|,|x_2|\le N\\ \left(x_1^2+\alpha_2x_2^2,q\right)=1\\  (x_3,q)=1}} \Phi(x_3).
$$
Proceeding along the lines of \cite[section 4]{BaCh1}, we obtain the asymptotic formula
$$
M(q)=C_q\cdot \frac{(2N)^2}{q}\cdot \int\limits_{\mathbb{R}} \Phi(x){\rm d}x+O\left(\frac{N^2}{q^{1-\varepsilon}}\right)
$$
if gcd$(2\alpha_2,q)=1$.
The only difference to \cite[section 4]{BaCh1} is that we here have a smoothed sum in $x_3$, for which we easily show that 
$$
\sum\limits_{\substack{x_3\in \mathbb{Z}\\ (x_3,q)=1}} \Phi(x_3)= \frac{\varphi(q)}{q}\cdot \int\limits_{\mathbb{R}} \Phi(x){\rm d}x+O(q^{\varepsilon})
$$
using the properties of $\Phi$. Now the result of Theorem \ref{mainthm2} follows if we can establish that
\begin{equation*}
V:=\sum\limits_{\substack{Q<q\le 2Q\\ \text{gcd}(2\alpha_2,q)=1}} \sum\limits_{\substack{\alpha_3 \bmod{q}\\ \text{gcd}(\alpha_3,q)=1}} \left|S(q,\alpha_3)-M(q)\right|^2 \ll N^6Q^{-3\varepsilon}.
\end{equation*}

As in \cite[section 3]{BaCh1}, we use Dirichlet characters to detect the congruence condition  $x_1^2+\alpha_2x_2^2+\alpha_3x_3^2\equiv 0 \bmod{q}$, arriving at 
\begin{equation} \label{Vsplit}
V=\sum\limits_{\substack{Q<q\le 2Q\\ \text{gcd}(2\alpha_2,q)=1}} V_1(q)+\sum\limits_{\substack{Q<q\le 2Q\\ \text{gcd}(2\alpha_2,q)=1}} V_2(q)
\end{equation}
after a short calculation, where 
\begin{equation}\label{s1}
    V_1(q):=\frac{1}{\varphi(q)} \cdot  \sum\limits_{\substack{q_1|\text{rad}(q)\\ q_1>1}} \bigg|\sum_{\substack{|x_1|,|x_2|\leq N\\ (x_1^2+\alpha_2x_2^2,q_2)=1}} \left(\frac{x_1^2+\alpha_2x_2^2}{q_1}\right)\bigg|^2\cdot \bigg|\sum\limits_{\substack{x_3\in \mathbb{Z}\\ (x_3,q)=1}} \Phi(x_3)\bigg|^2 
\end{equation}
with $q_1q_2=\mbox{rad}(q)$, and 
\begin{equation}\label{s2}
    V_2(q):=\frac{1}{\varphi(q)} \sum\limits_{\substack{\chi\bmod{q}\\ \chi^2\neq \chi_0}}\bigg|\sum_{|x_1|,|x_2|\leq N} \chi\left(x_1^2+\alpha_2x_2^2\right) \bigg|^2 \cdot \bigg| \sum_{x_3\in \mathbb{Z}} \Phi(x_3)\overline{\chi}^2(x_3)\bigg|^2.
\end{equation}    
In \cite[sections 6 and 7]{BaCh1}, we demonstrated that (with 1 in place of the summand $\Phi(x_3)$ and $\varepsilon$ in place of $3\varepsilon$)
$$
V_1(q)\ll N^6q^{-1-3\varepsilon},
$$
provided that $N\ge q^{1/3+45\varepsilon}$. So for the proof of Theorem \ref{mainthm3}, it is enough to establish that 
\begin{equation} \label{newgoal}
\sum\limits_{\substack{Q<q\le 2Q\\ \text{gcd}(2\alpha_2,q)=1}} V_2(q)\ll N^6Q^{-3\varepsilon}.
\end{equation}

\section{Elementary estimates}
Using H\"older's inequality, we have 
\begin{equation} \label{Hold}
\sum\limits_{\substack{Q<q\le 2Q\\ \text{gcd}(2\alpha_2,q)=1}} V_2(q)\le E_1^{1/2}E_2^{1/4}F^{1/4},
\end{equation}
where 
\begin{equation} \label{Ekdef}
E_k:=\sum_{\substack{Q< q\leq 2Q\\ \text{gcd}(2\alpha_2,q)=1}}\frac{1}{\varphi(q)}\sum\limits_{\chi \bmod{q}}\bigg|\sum_{|x_1|,|x_2|\leq N}\chi(x_1^2+\alpha_2x_2^2)\bigg|^{2k} \quad \mbox{for } k=1,2
\end{equation}
and 
$$
F:=\sum_{\substack{Q< q\leq 2Q\\ \text{gcd}(2\alpha_2,q)=1}} \frac{1}{\varphi(q)}\sum\limits_{\substack{\chi \bmod{q}\\ \chi^2\neq \chi_0}}\bigg|\sum_{x_3\in \mathbb{Z}}\Phi(x_3)\chi^2(x_3)\bigg|^8,
$$
where we have replaced $\overline{\chi}$ by $\chi$. 
The following two lemmas provide bounds for $E_1$ and $E_2$. 

\begin{Lemma} \label{E1} For $Q,N\ge 1$, we have
$$
E_1\ll_{\varepsilon} (N^2+Q)N^{2+\varepsilon}.
$$
\end{Lemma}

\begin{proof}
This follows from \cite[equations (19) and (20)]{BaCh1}.  
\end{proof}

\begin{Lemma} \label{E2} For $Q,N\ge 1$, we have
$$
E_2\ll_{\varepsilon} (N^4+Q)N^{4}(\alpha_2N)^{\varepsilon}.
$$
\end{Lemma}

\begin{proof}
Using the orthogonality relations for Dirichlet characters, we obtain
\begin{equation*}
E_2
=\sum_{\substack{Q<q\leq 2Q\\ \text{gcd}(2\alpha_2,q)=1}}\sum\limits_{\substack{|x_1|,|x_2|,|y_1|,|y_2|\leq N\\|u_1|,|u_2|,|v_1|,|v_2|\leq N\\ \text{gcd}((x_1^2+\alpha_2x_2^2)(y_1^2+\alpha_2 y_2^2)(u_1^2+\alpha_2u_2^2)(v_1^2+\alpha_2 v_2^2),q)=1\\ (x_1^2+\alpha_2x_2^2)(y_1^2+\alpha_2 y_2^2)\equiv (u_1^2+\alpha_2u_2^2)(v_1^2+\alpha_2 v_2^2)\bmod{q}}}1. 
\end{equation*}
Splitting this into two sums according to whether $$(x_1^2+\alpha_2x_2^2)(y_1^2+\alpha_2 y_2^2)= (u_1^2+\alpha_2u_2^2)(v_1^2+\alpha_2 v_2^2)$$ or not, and removing some summation conditions, we have 
$$
E_2\le G_0+G_1,
$$
where
$$
G_0:=\sum_{Q<q\leq 2Q}\sum\limits_{\substack{|x_1|,|x_2|,|y_1|,|y_2|\leq N\\|u_1|,|u_2|,|v_1|,|v_2|\leq N\\(x_1^2+\alpha_2x_2^2)(y_1^2+\alpha_2 y_2^2)= (u_1^2+\alpha_2u_2^2)(v_1^2+\alpha_2 v_2^2)}}1 
$$
and 
\begin{equation} \label{G1def}
G_1:=\sum\limits_{\substack{|x_1|,|x_2|,|y_1|,|y_2|\leq N\\|u_1|,|u_2|,|v_1|,|v_2|\leq N\\(x_1^2+\alpha_2x_2^2)(y_1^2+\alpha_2 y_2^2)\neq (u_1^2+\alpha_2u_2^2)(v_1^2+\alpha_2 v_2^2)}}\sum\limits_{\substack{q\in \mathbb{N}\\ q | \left((x_1^2+\alpha_2x_2^2)(y_1^2+\alpha_2 y_2^2)-(u_1^2+\alpha_2u_2^2)(v_1^2+\alpha_2 v_2^2)\right)}}1,
\end{equation}
where we have interchanged summations and then dropped the condition $Q<q\le 2Q$ in $G_1$. 
Recall that $\alpha_2$ is assumed to be a positive integer. Applying the divisor bound $\tau(n)\ll_\varepsilon n^{\varepsilon}$ gives
$$
G_1\ll_{\varepsilon} N^8(\alpha_2N)^{\varepsilon}.
$$  
We may bound $G_0$ by
$$
G_0\le Q \sum\limits_{|x_1|,|x_2|,|y_1|,|y_2|\le N} \sum\limits_{\substack{d,e\in \mathbb{N}\\ (x_1^2+\alpha_2x_2^2)(y_1^2+\alpha_2 y_2^2)=de}} r(d,\alpha_2)r(e,\alpha_2),
$$ 
where $r(n,\alpha)$ is the number of representations of $n\in \mathbb{N}$ in the form $n=u^2+\alpha_2v^2$ with $u,v\in \mathbb{Z}^2$. Now using Proposition \ref{repr} in conjunction with the divisor bound $\tau(n)\ll_{\varepsilon} n^{\varepsilon}$ gives
$$
G_0\ll_{\varepsilon} QN^4(\alpha_2N)^{\varepsilon}.
$$
Combining everything above yields the desired estimate for $E_2$. 
\end{proof}

Noting that $Q^{1/4}\le N\le Q^{1/2}$ in Theorem \ref{mainthm2}, the bounds in \eqref{Hold} and Lemmas \ref{E1} and \ref{E2} imply that
\begin{equation} \label{Holdnew}
\sum\limits_{Q<q\le 2Q} V_2(q)\ll_{\varepsilon} Q^{1/2}N^3F^{1/4}(\alpha_2N)^{\varepsilon}.
\end{equation} 
Moreover, since for every Dirichlet character $\chi_1$ modulo $q$, the number of characters $\chi$ modulo $q$ satisfying $\chi^2=\chi_1$ is $\ll \tau(q)\ll_{\varepsilon} q^{\varepsilon}$, we have the bound 
\begin{equation} \label{Fbound}
F\ll_{\varepsilon} Q^{\varepsilon} \sum_{Q< q\leq 2Q}\frac{1}{\varphi(q)}\sum\limits_{\substack{\chi \bmod{q}\\ \chi\neq \chi_0}}\bigg|\sum_{n\in \mathbb{Z}}\Phi(n)\chi(n)\bigg|^8,
\end{equation}
where we have dropped the summation condition $\text{gcd}(2\alpha_2,q)=1$ for simplicity. 
\section{Estimation of the main sum}
To prove Theorem \ref{mainthm2}, it remains to estimate the eighth moment of a short smoothed character sums on the right-hand side of \eqref{Fbound}. For $n>0$, we write the smooth weighth function $\Phi(n)$ using the Mellin inversion formula, Proposition \ref{MIF}, in the form 
\begin{equation*}
\Phi(n)=\frac{1}{2\pi i}\int_{c-i\infty}^{c+i\infty}\hat{\Phi}(s)n^{-s}ds,
\end{equation*}
where $c>1$ and 
$$
\hat{\Phi}(s):= \int_{0}^{\infty} \Phi(x) x^{s-1}{\rm d}x
$$
is the Mellin transform of $\Phi$. 
Summing over $n\in \mathbb{N}$, we conclude that
\begin{equation*}
\sum_{n=1}^{\infty}\Phi(n)\chi(n)=\frac{1}{2\pi i} \int_{c-i\infty}^{c+i\infty}\hat{\Phi}(s)L(s,\chi)ds.  
\end{equation*}
The integrand is analytic in the half plane $\mbox{Re}(s)>0$, and hence, we may use Cauchy's integral theorem to write
$$
\int_{c-i\infty}^{c+i\infty} \hat{\Phi}(s)L(s,\chi)ds=(I_1+I_2+I_3+I_4+I_5)\hat{\Phi}(s)L(s,\chi)ds,
$$ 
where 
$$
I_1:=\int_{c-i\infty}^{c-iT}, \quad I_2:=\int_{c-IT}^{1/2-iT}, \quad I_3:=\int_{1/2-iT}^{1/2+iT}, \quad I_4:=\int_{1/2+iT}^{c+iT}, \quad I_5:=\int_{c+iT}^{c+i\infty}
$$
for any $T>0$. Performing integration by parts for $j$ times and recalling the properties of $\Phi$ in Theorem \ref{mainthm2} gives the bound 
\begin{equation} \label{mellin1}
\hat\Phi(\sigma+it)\ll_j |t|^{-j} N^{\sigma-1+j}\int\limits_{\mathbb{R}} |\Phi^{(j)}(x)|dx\ll \Delta N^{\sigma-1}\cdot \left(\frac{N}{\Delta|t|}\right)^{j}, 
\end{equation}
which holds uniformly for $\sigma\ge 1/2$. So taking $j=\lceil 2028/\varepsilon\rceil$ and recalling that $N\le Q^{1/2}$, we get 
\begin{equation} \label{neglig}
\hat{\Phi}(\sigma+it)\ll Q^{-2027} \quad \mbox{for } \frac{1}{2}\le \sigma\le c \mbox{ and } |t|\ge T
\end{equation}
if 
\begin{equation} \label{T}
c:=1+\frac{1}{\log(2N)} \quad \mbox{and} \quad T:=\frac{Q^{\varepsilon}N}{\Delta}.
\end{equation}
We will keep these choices of $c$ and $T$ in the following. Using \eqref{neglig} together with Proposition \ref{CB}, it follows that 
$$
I_2,I_4\ll Q^{-2026}.
$$
(Finally, we are down to this year.) Similarly, using \eqref{mellin1}, we calculate that
$$
I_1,I_5\ll Q^{-2026}.
$$
We also have the trivial bound
\begin{equation} \label{mellin2}
\hat\Phi(\sigma+it)\ll N^{\sigma}, 
\end{equation}
and hence, 
$$
I_3\ll N^{1/2}\int\limits_{-T}^T \left|L\left(\frac{1}{2}+it,\chi\right)\right|{\rm d}t.
$$
Combining the above, we arrive at the bound
$$
\sum\limits_{n=1}^{\infty} \Phi(n)\chi(n) \ll N^{1/2}\int\limits_{-T}^T \left|L\left(\frac{1}{2}+it,\chi\right)\right|{\rm d}t+Q^{-2025}. 
$$
The same bound can be derived along similar lines for the sum $\sum\limits_{n<0}\Phi(n)\chi(n)$, and hence, using H\"older's inequality for integrals and \eqref{Fbound}, we deduce that
$$
F\ll Q^{\varepsilon} T^7 N^4\sum_{Q< q\leq 2Q}\frac{1}{\varphi(q)} \sum\limits_{\chi \bmod{q}}\int\limits_{-T}^{T}  \left|L\left(\frac{1}{2}+it,\chi\right)\right|^8 {\rm d}t+Q^{-2025}.
$$
Finally, applying Proposition \ref{eighth} to the right-hand side above, and recalling the choice of $T$ in \eqref{T} and $N\le Q^{1/2}$, we obtain 
\begin{equation} \label{finalFbound}
F\ll (QT)^{2\varepsilon}T^9N^4Q\ll \left(\frac{N}{\Delta}\right)^9N^4 Q^{1+12\varepsilon}
\end{equation}
if $\varepsilon$ is small enough.

\section{Proofs of Theorems \ref{mainthm2} and \ref{mainthm3}} \label{final}
Now we are ready to finalize the proof of Theorem \ref{mainthm2}. Combining \eqref{Holdnew} and \eqref{finalFbound}, and using $N\le Q^{1/2}$, we have 
\begin{equation} 
\sum\limits_{Q<q\le 2Q} V_2(q)\ll_{\varepsilon} Q^{3/4}N^4\cdot \left(\frac{N}{\Delta}\right)^{9/4}(\alpha_2Q)^{4\varepsilon}.
\end{equation} 
The desired bound \eqref{newgoal} follows, provided that 
$$
\Delta\ge Q^{1/3}N^{1/9}(\alpha_2 Q)^{4\varepsilon}.
$$
This completes the proof of Theorem \ref{mainthm2}. 

We keep the proof of Theorem \ref{mainthm3} brief since many arguments repeat. First, we approximate the sharp cutoff sum in \eqref{newasym} by a smoothed sum as in \eqref{asympform'}, where we assume that
\begin{equation} \label{newDeltaN}
\Delta\le NQ^{-2\varepsilon}.
\end{equation}  
Accordingly, the expected main term changes by a small quantity of size $$\ll N^3\varphi(q)^{-1}Q^{-2\varepsilon}.$$
Here we note that $\Phi(x)$ agrees with the characteristic function of the interval $[-N,N]$ except for $N-\Delta\le |x|\le N+\Delta$.   
Now the claimed asymptotic \eqref{asympform'} follows for all $(q,\alpha_3)\in \mathcal{B}(Q,\alpha_2)$ with at most $O\left(Q^{2-\varepsilon}\right)$ exceptions if $\Delta$ and $N$ satisfy the conditions \eqref{DeltaN}, \eqref{newDeltaN} and 
\begin{equation} \label{anothergoal}
\sum\limits_{\substack{Q<q\le 2Q\\ \text{gcd}(2\alpha_2,q)=1}} \bigg| \sum\limits_{\substack{|x_1|,|x_2|\le N\\ N-\Delta\le x_3\le N+\Delta\\ (x_3,q)=1\\ 
x_1^2+\alpha_2x_2^2+\alpha_3x_3^2\equiv 0 \bmod{q}}} 1\bigg|^2 \ll N^6Q^{-3\varepsilon}.
\end{equation}
Following similar arguments as in section \ref{ini}, \eqref{anothergoal} holds if  
\begin{equation} \label{newgoal'}
\sum\limits_{\substack{Q<q\le 2Q\\ \text{gcd}(2\alpha_2,q)=1}} V_2'(q)\ll N^6Q^{-3\varepsilon},
\end{equation}
where 
$$
    V_2'(q):=\frac{1}{\varphi(q)} \sum\limits_{\substack{\chi\bmod{q}\\ \chi^2\neq \chi_0}}\bigg|\sum_{|x_1|,|x_2|\leq N} \chi\left(x_1^2+\alpha_2x_2^2\right) \bigg|^2 \cdot \bigg| \sum_{N-\Delta\le |x_3|\leq N+\Delta} \overline{\chi}^2(x_3)\bigg|^2.    
$$
Now we simply estimate the short character sum over $x_3$ using Proposition \ref{Burgess} with $r=3$, obtaining
$$
\sum\limits_{\substack{Q<q\le 2Q\\ \text{gcd}(2\alpha_2,q)=1}} V_2'(q)\ll \Delta^{4/3} Q^{2/9+\varepsilon}E_1,
$$  
where $E_1$ is defined as in $\eqref{Ekdef}$. Using the condition $N\le Q^{1/2}$ and Lemma \ref{E1}, it follows that
$$
\sum\limits_{\substack{Q<q\le 2Q\\ \text{gcd}(2\alpha_2,q)=1}} V_2'(q)\ll \Delta^{4/3} Q^{11/9+2\varepsilon}N^2.
$$ 
This is less than or equal to $N^6Q^{-3\varepsilon}$ if  
$$
\Delta\le  Q^{-11/12-4\varepsilon}N^3.
$$
We fix 
$$
\Delta:=\min\left\{Q^{-11/12-4\varepsilon}N^3,Q^{-2\varepsilon}N\right\}. 
$$
Then \eqref{DeltaN} and \eqref{newDeltaN} hold if 
$$
Q^{45/104}(\alpha_2 Q)^{3\varepsilon}\le N\le Q^{1/2}.
$$
This completes the proof of Theorem \ref{mainthm3}. \\ \\
{\bf Remark 2:} If we take $r=4$ in place of $r=3$ in our application of Proposition \ref{Burgess} above, then we obtain the slightly better exponent $159/368$ in place of $45/104$. However, in this case, the said Proposition \ref{Burgess} is applicable only if $q$ is cubefree. Taking larger values $r>4$ in Proposition \ref{Burgess} does not improve the result.

\end{document}